\documentclass[a4paper]{amsart}
\usepackage{lmodern}
\usepackage{amsthm}
\usepackage{amsmath, amsthm}
\usepackage{amssymb}
\usepackage{hyperref}
\hypersetup{
  unicode=true,
  colorlinks=true,
  citecolor=blue,
  linkcolor=blue,
  anchorcolor=blue
}
\usepackage{todonotes}
\usepackage{tikz-cd}

\usepackage[pagewise,mathlines]{lineno}

\usepackage[nameinlink,capitalise,noabbrev]{cleveref}

\newtheorem{theorem}{Theorem}[section]
\newtheorem{claim}[theorem]{Claim}
\newtheorem{lemma}[theorem]{Lemma}
\newtheorem{remark}[theorem]{Remark}
\newtheorem{proposition}[theorem]{Proposition}
\newtheorem{corollary}[theorem]{Corollary}
\newtheorem{fact}[theorem]{Fact}
\newtheorem{definition}[theorem]{Definition}
\newtheorem{question}[theorem]{Question}

\DeclareMathOperator{\len}{len}
\DeclareMathOperator{\id}{id}
\DeclareMathOperator{\crit}{crit}
\DeclareMathOperator{\Ult}{Ult}
\DeclareMathOperator{\Col}{Col}
\DeclareMathOperator{\otp}{otp}
\DeclareMathOperator{\cf}{cf}
\DeclareMathOperator{\acc}{acc}
\DeclareMathOperator{\Hull}{Hull}
\DeclareMathOperator{\image}{``}
\DeclareMathOperator{\stem}{stem}
\DeclareMathOperator{\sym}{sym}
\DeclareMathOperator{\fix}{fix}

\newcommand{\forces}{\mathrel{\Vdash}}

\newcommand{\ZF}{\mathsf{ZF}}
\newcommand{\ZFC}{\mathsf{ZFC}}
\newcommand{\GCH}{\mathsf{GCH}}
\newcommand{\DC}{\mathsf{DC}}
\newcommand{\HS}{\mathsf{HS}}
\newcommand{\AD}{\mathsf{AD}}
\newcommand{\Ord}{\mathrm{Ord}}
\newcommand{\HOD}{\mathrm{HOD}}
\newcommand{\Reg}{\mathrm{Reg}}

\newcommand{\power}{\mathcal P}

\newcommand{\cW}{\mathcal W}
\newcommand{\cU}{\mathcal U}
\newcommand{\cF}{\mathcal F}

\newcommand{\cG}{\mathcal G}

\def\ssubset{\mathrel{\vtop{\ialign{##\crcr$\hfil\subset\hfil$\crcr\noalign{\kern1.5pt\nointerlineskip}$\hfil\widetilde{}\hfil$\crcr\noalign{\kern1.5pt}}}}\vspace*{-0.3\baselineskip}}

\date{11 December, 2024}
\title{The first measurable can be the first inaccessible cardinal}
\author[M.~Gitik]{Moti Gitik}
\address[Moti Gitik]{School of Mathematical Sciences \\
Tel Aviv University \\
Tel Aviv 69978 \\
Israel}
\email{gitik@tauex.tau.ac.il}
\author[Y.~Hayut]{Yair Hayut}
\address[Yair Hayut]{Einstein Institute of Mathematics \\
Edmond J. Safra Campus \\
The Hebrew University of Jerusalem \\
Givat Ram. Jerusalem, 9190401, Israel}
\email{yair.hayut@mail.huji.ac.il}
\author[A.~Karagila]{Asaf Karagila}
\address[Asaf Karagila]{School of Mathematics\\
  University of Leeds\\
  Leeds, LS2~9JT, UK}
\email{karagila@math.huji.ac.il}

\begin{document}
\begin{abstract}
In \cite{SmallMeasurableCardinals} the second and third authors showed that if the least inaccessible cardinal is the least measurable cardinal, then there is an inner model with $o(\kappa)\geq2$. In this paper we improve this to $o(\kappa)\geq\kappa+1$ and show that if $\kappa$ is a $\kappa^{++}$-supercompact cardinal, then there is a symmetric extension in which it is the least inaccessible and the least measurable cardinal.
\end{abstract}
\maketitle
\section{Introduction}
Large cardinal axioms form the yardstick with which we measure the consistency strength of various mathematical statements. In other words, given a mathematical statement, we can use large cardinals to give lower and upper bounds as to the question ``how strong of a mathematical foundation is required to prove the statement is possibly true?''. Perhaps the most famous large cardinal axiom is the one positing the existence of an inaccessible cardinal, or a ``Tarski--Grothendieck universe''.\footnote{Without the axiom of choice the many definitions of an inaccessible cardinal which are equivalent in $\ZFC$ will no longer need to be equivalent (see \cite{Inaccessible} for details). In this work ``$\kappa$ is inaccessible'' means that $V_\kappa\models\ZF_2$, that is $\ZF$ formulated in second-order logic, or equivalently there is no $x\in V_\kappa$ and a function $f\colon x\to\kappa$ whose image is a cofinal subset of $\kappa$.}

Measurable cardinals, in the standard context of set theory, where the axiom of choice is taken as true, can be defined by two equivalent formulations: the existence of ultrafilters;\footnote{A cardinal $\kappa$ is measurable if and only if there is a non-principal $\kappa$-complete ultrafilter on $\kappa$.} or as the critical points of elementary embeddings. The equivalence, which relies heavily on {\L}o\'s' theorem, can fail without the axiom of choice. In the 1960s, Jech proved that $\omega_1$, the least uncountable cardinal which can never be a critical point of an elementary embedding,\footnote{At least not if we require the embedding to be definable or at least amenable, that is.} could be a measurable cardinal \cite{Jech}.

A sequence of results under the axiom of determinacy, starting with Solovay's proof of the measurability of $\omega_1$ and $\omega_2$, and reaching its final form in Steel's theorem that assuming $V=L(\mathbb R)$, every uncountable regular cardinals below $\Theta$ is measurable (\cite[Theorem~8.27]{SteelHandbook}), shows that measurable cardinals are common in some natural models of $\ZF$.

In \cite{CriticalCardinals}, the second and third authors isolated the notion of a ``critical cardinal'' which is a critical point of an elementary embedding, and studied some of the consequences of the existence of critical cardinals without the axiom of choice.

That work led to the question of how small can the least measurable cardinal be, if the axiom of choice is allowed to fail. Of course, it can be the least uncountable cardinal, but that is not a satisfying answer to the question. Specifically, we want to understand this phenomenon within the context of the large cardinal hierarchy, whereas $\omega_1$ is a successor cardinal. In \cite{SmallMeasurableCardinals} the two authors show that the least Mahlo cardinal\footnote{This is ``the next step'' after inaccessibility, in the sense that the set of inaccessible cardinals below a Mahlo cardinal is not negligible, i.e.\ \emph{stationary}.} can be the least measurable cardinal as well, and that the large cardinal strength of this assertion is merely ``there exists a measurable cardinal'', that is to say that in order to produce a model where the least measurable cardinal is the least Mahlo cardinal, it was enough to start with a model in which a single measurable cardinal exists. However, in trying to reduce the measurability even lower we run into an intriguing situation. If the least measurable cardinal is also the least inaccessible cardinal, then we must have began with a model with many measurable cardinals.\footnote{In the previous paper the lower bound was $o(\kappa)\geq 2$.}

This means that producing a model where the least measurable cardinal is the least inaccessible cardinal requires us to work harder and start with stronger assumptions. The difficulty does not lie in the fact that this the first inaccessible, but rather in the fact that the cardinal is not Mahlo. Once the set of inaccessible, or even regular, cardinals is negligible (namely, non-stationary), all manner of difficulties start to arise. This phenomenon, and therefore the question that we are concerned with here, is not unique to situations where the axiom of choice fails. For example, even in $\ZFC$ it is not known if the tree property can hold at a non-Mahlo weakly inaccessible cardinal, whether it is consistent that there is no Suslin tree on a non-Mahlo inaccessible cardinal, and it is known that the failure of diamond principles at the least inaccessible cardinal also has a perhaps-surprisingly strong large cardinal lower bound.

If one replaces a (strongly) inaccessible by a weakly inaccessible, then starting from $\AD+V=L(\mathbb R)$, Apter constructed a model of $\ZF$ in which the least measurable and the least weakly inaccessible coincide (see \cite{Apt}). Namely, basing on the aforementioned result of Steel \cite{SteelHandbook}, he uses the Prikry forcing (in a $\ZF$ context) to turn any given set $A\subseteq \Theta$ of measurable cardinals into singulars and preserving measurability of the rest. However, weakly inaccessible cardinals, in particular in the context of $\ZF$, tend to feel quite unsatisfactory for the same reasons that $\omega_1$ does not answer our question.

\subsection{In this paper}
In this work we establish an upper bound for the statements ``there is a measurable cardinal that is not a Mahlo cardinal'' and ``the least inaccessible cardinal is the least measurable cardinal'', as well as far improved lower bounds. Indeed, we show that if $\kappa$ is a $\kappa^{++}$-supercompact cardinal, then there is a symmetric extension in which it is a non-Mahlo measurable cardinal, and a further symmetric extension in which it is also the least inaccessible cardinal. In both of these, we show that $\DC_{<\kappa}$ holds, and $\ZFC$ holds below $\kappa$. We also show that the lower bound required for these results is at least as high as ``there is an inner model with $o(\kappa)\geq\kappa+1$''.

We use supercompact Radin forcing to construct the symmetric extension, and the Mitchell covering lemma (for ``there is no inner model with $o(\alpha)=\alpha^{++}$'') to provide the lower bounds. Some questions and conjectures are given at the end as well.

\subsection{Technical preliminaries}
We assume that the readers are familiar with the techniques of forcing and symmetric extensions, but we include a brief outline of the latter. Fixing a forcing notion $\mathbb P$, an automorphism, $\pi$, of $\mathbb P$ acts on the $\mathbb P$-names in a recursive definition given by \[\pi\dot x=\{\langle\pi p,\pi\dot y\rangle\mid\langle p,\dot y\rangle\in\dot x\}.\] Unsurprisingly, as $\forces$ is defined from the order, $p\forces\varphi(\dot x)$ if and only if $\pi p\forces\varphi(\pi\dot x)$.

Fixing a group of automorphisms, $\cG$, we define $\sym(\dot x)$, for a $\mathbb P$-name $\dot x$, as $\{\pi\in\cG\mid\pi\dot x=\dot x\}$. Since interpretation of names by a generic filter, $G$, satisfies the equation \[(\pi^{-1}\dot x)^G=\dot x^{\pi\image G},\] we get that $\dot x^G=\dot x^{\pi\image G}$ when $\pi\dot x=\dot x$.

We want, therefore, to isolate a notion which allows us to say when a $\mathbb P$-name is interpreted the same way by ``most'' generic filters. Towards that goal, $\cF$ is a \emph{filter of subgroups} if it is a non-empty collection of subgroups of $\cG$ which is closed under supergroups and finite intersections. We say that it is a \emph{normal} filter of subgroups if whenever $H\in\cF$ and $\pi\in\cG$, then $\pi H\pi^{-1}\in\cF$ as well.\footnote{The clash in terminology with normal filters in the sense of diagonal intersections is known. We will refer to ultrafilters on ordinals as measures whenever any confusion may arise.}

Let a $\mathbb P$-name, $\dot x$, be called \emph{symmetric} if $\sym(\dot x)\in\cF$, and if this property holds hereditarily for all names appearing in $\dot x$, we say that it is \emph{hereditarily symmetric}. The class of hereditarily symmetric names is denoted by $\HS$ and if $G$ is a generic filter, $\HS^G=\{\dot x^G\mid\dot x\in\HS\}$ is a transitive model of $\ZF$ intermediate between $V$ and $V[G]$, and we refer to it as a \emph{symmetric extension}.

We say that a class $M$ is \emph{$\kappa$-closed} (in $V$, or in a larger class) if whenever $\gamma<\kappa$ and $f\colon\gamma\to M$, then $f\in M$. We say that a forcing notion is $\kappa$-closed if every descending sequence of fewer than $\kappa$ conditions has a lower bound, and a tree is $\kappa$-closed if it is $\kappa$-closed as a forcing notion (which necessitate reversing its order in our convention). In our context, $\langle T,<\rangle$ is a tree if it is a well-founded partial order with a minimum element, such that whenever $t_0,t_1<t$ we have that $t_0$ and $t_1$ are comparable.

In these symmetric extensions we often want to preserve some fragments of the axiom of choice, and in this work we will be primarily focused on \emph{Dependent Choice}. Amongst the many interesting equivalences, $\DC_\kappa$ can be formulated as ``Every $\kappa$-closed tree has a maximal node or a chain of type $\kappa$'', and $\DC_{<\kappa}$ means that for every $\lambda<\kappa$, $\DC_\lambda$ holds.

For a set of ordinals, $x$, we write $\acc x$ to denote its accumulation points, i.e.\ $\{\alpha\mid\alpha=\sup\alpha\cap x\}$. We denote by $\pi_x$ the Mostowski collapse of $x$, i.e.\ the order isomorphism with its order type. And we say that $\xi$ is a successor in $x$ if $\pi_x(\xi)$ is a successor ordinal. For a set of ordinals $x \subseteq \lambda$, we can naturally extend $\pi_x$ to $\mathcal{P}(\lambda)$ and even $\power(\power(\lambda))$. For $a\in\power(\lambda)$ we define $\pi_x(a) = \pi_x\image(x\cap a)$. For $A\in\power(\power(\lambda))$, we let $\pi_x(A) = \pi_x\image A$.

Finally, for $x,y\in\power_\kappa\lambda$ we write $x\ssubset y$ to mean $x\subseteq y$ and $|x|<|y\cap\kappa|$.\footnote{In our context, $y\cap\kappa$ will be a cardinal.}

\section{Measurable, but not Mahlo}\label{section:upper-bound}
This section will be devoted for the proof of the following theorem.
\begin{theorem}[$\ZFC$]\label{thm:non-mahlo}
Assume $\GCH$ and let $\kappa$ be $\kappa^{++}$-supercompact. Then, there is a symmetric extension in which $\kappa$ is non-Mahlo inaccessible and there is a normal measure on $\kappa$.
\end{theorem}
The idea of the proof of this theorem is to work in an intermediate model between the generic extension by the supercompact Radin club and the standard Radin club. Let us begin by recalling some basic facts about the supercompact Radin forcing and its connection (in this specific case) with the standard Radin forcing.
\subsection{Radin Forcing}
We will follow Krueger's presentation of supercompact Radin forcing (\cite{Krueger}). This presentation translates the classical construction of Foreman and Woodin (\cite{ForemanWoodin}) to the context of coherent sequences of supercompact measures, for which we have a clear presentation of the relevant automorphisms.

Since $\kappa$ is $\kappa^{++}$-supercompact, by Proposition~2.2 of \cite{Krueger}, there are sequences $o^\cW \colon \kappa+1 \to \Ord$, with $o^\cW(\kappa) = \kappa^{++}$, and a \emph{coherent sequence} of measures, $\cW = \langle W(\alpha, i) \mid \alpha \leq \kappa, i < o^{\cW}(\alpha)\rangle$. In our context this means that:\footnote{In \cite{Krueger} two additional conditions are required that hold automatically in our context.}
\begin{enumerate}
\item For every $\alpha \leq \kappa$ and $i < o^{\cW}(\alpha)$, $W(\alpha, i)$ is a normal measure on $\power_\alpha\alpha^+$.
\item For every $\alpha \leq \kappa$ and $i < o^{\cW}(\alpha)$, \[ j_{W(\alpha,i)}(\cW)(\alpha) = \cW(\alpha)\restriction i=\langle W(\alpha,k)\mid k<i\rangle,\]
where $\mathcal{W}(\alpha):=\langle W(\alpha,i)\mid i < o^{\mathcal{W}}(\alpha)\rangle$.
\end{enumerate}
Given a coherent sequence $\cW$ the \emph{supercompact Radin forcing}, $\mathbb{R}(\cW)$, as defined by Krueger, see \cite[Section~3]{Krueger}. For the completeness of the paper we will provide a definition of the forcing in our specific case.

A condition $p\in \mathbb{R}(\mathcal{W})$ is a finite sequence of the form $\langle d_0, d_1,\dots, d_{n-1}, A\rangle$ where:
\begin{enumerate}
\item The set $A$ lies in $\bigcap_{\zeta < o^{\cW}(\kappa)} \cW(\kappa, \zeta)$,
\item $n \geq 0$, and for every $i < n$, $d_i$ is either a member of $\power_\kappa\kappa^{+}$ or a pair of the form $\langle a_i, A_i\rangle$.
\begin{itemize}
\item If $d_i \in \power_\kappa\kappa^{+}$ we set $a_i = d_i$ and $\rho_i = a_i\cap \kappa$ and require $\rho_i\in \kappa$.
\item Otherwise, $d_i = \langle a_i, A_i\rangle$, we set $\rho_i = a_i \cap \kappa$ and require $\rho_i \in \kappa$. Moreover, $A_i \in \bigcap_{\zeta < o^{\cW}(\rho_i)} \mathcal{W}(\rho_i, \zeta)$.
\end{itemize}
\end{enumerate}
For a condition $p$ we will denote by $a_i^p$, $d_i^p$ and $A_i^p$ the corresponding objects in $p$ and denote $n$ by $\len p$. We also set $\rho_n = \kappa$, $a_n^p = \kappa^{+}$ and $A_n = A$.

In order to define the extension of a condition, let us first define pure extension and one-point extension. We say that $p$ is a pure extension of $q$, or $p \leq^* q$ if $\len p = \len q$ and for every $i$, $a_i^p = a_i^q$.  For $\eta \in A_i$ for some $i\leq n$, we let $p^\smallfrown \eta$ (the one-step extension of $p$ by $\eta$) to be the condition defined as either
\begin{itemize}
\item $\langle d_0,\dots, d_{i-1}, \langle \pi_{a_i}^{-1}(\eta), \pi_\eta(A_i)\rangle, d_i,\dots, d_{n-1}, A\rangle$, if $o^{\cW}(\eta \cap \kappa) > 0$, or
\item $\langle d_0,\dots, d_{i-1}, \pi_{a_i}^{-1}(\eta), d_i,\dots, d_{n-1}, A\rangle$, if $o^{\cW}(\eta \cap \kappa) = 0$.
\end{itemize}
We say that $\eta$ is a legitimate extension, if $p^\smallfrown\eta$ is a condition. We let $p \leq q$ if $p$ can be obtained from $q$ using a sequence of legitimate one-point extensions and direct extensions. We will work under the implicit assumption that if $a\in\power_\alpha\alpha^+$, then $a\cap\alpha$ is a cardinal and $|a|=|a\cap\alpha|^+$, note that these sets form a large set in all of our measures, and so we can simply restrict our conditions to these sets. For the basic properties of the forcing, see \cite[Section~3]{Krueger}.

For a normal measure over $\power_\kappa \kappa^{+}$, $W$, let us denote by $W \restriction \kappa$ the projection of this measure to a normal measure on $\kappa$:
\[W \restriction \kappa = \{A \subseteq \kappa \mid \exists B \in W,\, A = \{x \cap \kappa \mid x \in B\}\}.\]
While this projection always induces a normal measure on $\kappa$, the coherence of the sequence of projections is more subtle.
\begin{lemma}\label{lemma:projection-coherence}
Assume $\GCH$. Let $\cW$ be a coherent sequence of $\power_\kappa\kappa^{+}$-supercompact measures with $o^{\cW}(\alpha) < \alpha^{++}$ for all $\alpha < \kappa$. Then, $\langle W(\alpha, i) \restriction \alpha \mid \alpha \leq \kappa, i < o^{\cW}(\alpha)\rangle$ is coherent.
\end{lemma}
\begin{proof}
Let $\iota$ be the ultrapower embedding by $W(\alpha,i)\restriction \alpha$, $j$ the ultrapower embedding by $W(\alpha, i)$, and $k$ be the quotient map. Namely, $k([f]) = j(f)(\alpha)$. We have the following commutative diagram
\begin{center}
\begin{tikzcd}
V\arrow[r, "\iota"] \arrow[rd, "j"]& N=\Ult(V, W(\alpha, i)\restriction\alpha)\arrow[d, "k"] \\
 & M=\Ult(V, W(\alpha, i))
\end{tikzcd}
\end{center}
The map $k$ must have a critical point (as for example the $V$-cofinality of $\iota(\alpha^+)$ is $\alpha^+$ while the $V$-cofinality of $j(\alpha^+)$ is $\alpha^{++}$). The critical point of $k$ must be an $N$-cardinal which is not in the image of $k$. Therefore $\crit k\geq\crit j=\alpha$. Equality, however, is impossible, as for $\id\colon \alpha\to \alpha$, $j(\id)(\alpha)=\alpha$, and it cannot be $\alpha^+$ as for $s\colon \alpha\to \alpha$, $\forall \zeta,\, s(\zeta) = \zeta^+$, $[s] = \alpha^+$ (using the fact that $N$ computes $\alpha^+$ correctly) and clearly $j(s)(\alpha) = \alpha^{+}$.

Therefore, $\crit k \geq (\alpha^{++})^N$. A simple computation shows that $|(\alpha^{++})^N|^V = \alpha^+$ and thus it has to move under $k$, so $\crit k = (\alpha^{++})^N$.

Since $o^{\cW}(\zeta) < \zeta^{++}$ for all $\zeta<\alpha$, we have that $\iota(o)(\alpha) = i < (\alpha^{++})^N = \crit k$. In particular, for every $\zeta < i$, \[k(\iota(W)(\alpha,\zeta) \restriction \alpha) = j(W)(\alpha, \zeta) \restriction \alpha = W(\alpha, \zeta) \restriction \alpha,\]
where the last equality follows from the fact that $\power(\alpha) \subseteq N$ and $\crit k > \alpha$.
\end{proof}
We will use $\overline\cW$ to denote the sequence of projected measures.
\begin{remark}
\cref{lemma:projection-coherence} does not make sense for longer sequences. Under $\GCH$, $2^{\kappa} = \kappa^{+}$ and thus there is no Mitchell increasing sequence of normal measures on $\kappa$ of length $> \kappa^{++}$, but a coherent sequence of measures on $\power_\kappa\kappa^+$ can be as long as $(2^{\kappa^{+}})^+ \geq \kappa^{+3}$.
\end{remark}

\begin{lemma}
In the Radin generic extension by $\mathbb{R}(\cW)$, $\kappa$ remains inaccessible.
\end{lemma}
The proof is standard, and the result itself appeared implicitly in the literature. We include it for the convenience of the reader.
\begin{proof}

First, let us derive a \emph{strong Prikry Property} from the standard Prikry Proprety. We will prove it only for conditions of length $0$ where the proof of the general case can be obtained using the factorization property (see \cite[Section~4]{Krueger}).

\begin{claim}
Let $p = \langle \varnothing, A\rangle$ be a condition of length $0$ and let $D$ be a dense open set in the supercompact Radin forcing. Then, there are:
\begin{enumerate}
\item a direct extension $q \leq^* p$,
\item a natural number $n < \omega$ and
\item a rooted tree of height $n$, $T \subseteq (\power_\kappa \kappa^+)^{\leq n}$, such that for each $\eta \in T$ with $|\eta|<n$, there is $\zeta_\eta < o^{\cW}(\kappa)$ such that $\{x \in \power_\kappa \kappa^{+} \mid \eta ^\smallfrown \langle x\rangle \in T\} \in W(\kappa,\zeta_\eta)$.\footnote{In particular, every maximal node in the tree is of height $n$.}
\end{enumerate}
such that for every $\eta$ in the top level of $T$, there is a direct extension of $q^\smallfrown \eta$ in $D$.
\end{claim}
\begin{proof}
By applying the Prikry Property and the $\sigma$-closure of the measures, we can conclude that there is $\langle\varnothing,A^q\rangle=q\leq^*p$ that decides the minimal length of a condition in the generic extension which is in $D$. Let $n$ be this length.

For every $x \in A^q$, let us check whether there is a direct extension of $q^\smallfrown \langle x\rangle$ that forces that there is a condition $r$ in the generic of length $n$, in $D$, such that $\min d^r = x$, and $x$ is minimal in the generic club with such property. Note that this set must be positive (with respect to the filter $\bigcap_{\zeta < o^{\cW}(\kappa)} W(\kappa,\zeta)$), as otherwise we could shrink $A^q$ to avoid it and get a contradiction. So, there is some $\zeta < o^{\cW}(\kappa)$ such that for a large set $T_0$ of $x$ with respect to $U_\zeta$, there is $q_x = \langle (x, A_0^x), A_1^x\rangle \leq^* q ^\smallfrown \langle x\rangle$ that forces the existence of a condition $r$ in $D \cap \dot{G}$ with length $n$ and  $\min d^r = x$.

For each such $x \in T_0$ we repeat the process, and find an ordinal $\zeta_x < o^{\cW}(\kappa)$ and measure one many $y$ (relative to $W(\kappa,\zeta_x)$) with the property that there is a direct extension of $q_x^\smallfrown \langle y\rangle$ forcing $\langle x,y\rangle$ to be the first two elements in $d^r$ for $r \in D \cap \dot{G}$ of length $n$. Continuing this way for $n$ steps we get the existence of $T$, which proves the claim.
\end{proof}

Let us consider the name $\dot{f}$ for a function from $\lambda$ to $\kappa$ for $\lambda < \kappa$. For each $\alpha < \lambda$, let $D_\alpha$ be the dense open set of all conditions deciding a value for $\dot{f}(\check\alpha)$. Applying our version of the strong Prikry Property for each $\alpha<\lambda$ we obtain a direct extension $q \leq^* p$ and a sequence of trees $\langle T_\alpha \mid \alpha < \lambda\rangle$ of various finite heights. We can attach to each one of the nodes of the trees $\eta \in T_\alpha$, the corresponding direct extension $q_\eta \leq q ^\smallfrown \eta$ from $D_\alpha$.

Let us consider all $\zeta_\eta$ for $\eta \in T$. As there are $\kappa^+$ such ordinals and $o^{\cW}(\kappa) = \kappa^{++}$, there is an ordinal $\zeta^* < \kappa^{++}$ bounding all of them.

In the ultrapower by $W(\kappa,\zeta^*)$ the we have $j\image T_\alpha$ for all $\alpha < \lambda$ as well as the corresponding $j(\eta) \mapsto j(q_\eta)$.

Consider $[\id]_{W(\kappa,\zeta^*)} = j\image \kappa^+$. It is easy to verify that one can add this element to each one of the $j(q_\eta)$ for each $\eta \in T_\alpha$.

Moreover, as all measures mentioned by the trees are below $\zeta^*$, for each $\alpha < \lambda$, $\{j(\eta) \in T_\alpha \mid \eta\text{ is a node}\}$ forms a maximal antichains in the Radin forcing below $j\image \kappa^+$.
Thus for each element in $T_\alpha$, $j(\eta)$, the condition $j(q)^\smallfrown \langle j(\eta), j \image \kappa^+\rangle = j(q)^\smallfrown \langle j\image \kappa^+, j(\eta)\rangle$ is forcing a value to $j(\dot f)(\check \alpha)$, as it extends the condition $j(q_\eta)$. By elementarity this value must be the $j$-image of the one that $q_\eta$ forced for $\dot{f}(\check\alpha)$ and thus below $\kappa$.

We conclude that $j(q)^\smallfrown \langle j\image \kappa^+\rangle$ forces $j(\dot{f})$ to have a range boudned by $\kappa$. Reflecting this, we obtain a $W(\kappa,\zeta^*)$-large set such that adding each $x$ in this set to the stem forces the range of $f$ to be bounded by $x \cap \kappa$.
\end{proof}
Let $G$ be generic for $\mathbb{R}(\cW)$. Let us denote the generic continuous and increasing sequence in $\power_\kappa\kappa^{+}$ by $C_G$, so \[C_G = \{x \mid \exists p \in G, x\in \stem p\}.\]

Let $\overline{C} = \{x \cap \kappa \mid x \in C_G\}$. Since $C_G$ is continuous and cofinal, $\overline{C}\subseteq \kappa$ is a club. To prove that $\overline{C}$ is a Radin club for $\mathbb{R}(\overline{\cW})$ we will need to use the Mathias Criterion for genericity.

Recall that a condition $p\in\mathbb{R}(\overline{\cW})$ is compatible with $\overline C\subseteq\kappa$ if $\stem p\subseteq\overline{C}$ and whenever $d_i<d_{i+1}$ are two successive points in $\stem p$, then $\overline{C}\cap(\alpha_i,\alpha_{i+1})\subseteq A_{i+1}$ if $d_{i+1}=\langle\alpha_{i+1},A_{i+1}\rangle$ or else $\overline{C}\cap(\alpha_i,\alpha_{i+1})=\varnothing$.

\begin{fact}[Mathias Criterion]
Let $\overline{C} \subseteq \kappa$ be a club. Let $\overline{G}$ be the collection of all conditions in $\mathbb{R}(\overline{\cW})$ compatible with $\overline{C}$. Then $\overline{G}$ is a generic filter iff
\begin{enumerate}
\item For every $\alpha \in \acc \overline{C}$, $\overline{C} \cap \alpha$ is generic for $\mathbb{R}(\overline{\cW} \restriction \alpha + 1)$.
\item For every $A \in\bigcap_{\zeta < o^{\cW}(\kappa)} \cW(\kappa, \zeta)$, there is $\eta < \kappa$ such that $\overline{C} \setminus \eta \subseteq A$.
\end{enumerate}
\end{fact}

\begin{lemma}
$\overline{C}$ is a generic Radin club for $\mathbb{R}(\overline{\cW})$.
\end{lemma}
\begin{proof}
This follows from the Mathias Criterion for genericity of the Radin club.
Indeed, let us prove by induction on $\alpha \in \acc \overline{C}$ that the criteria holds. Let $A \in \bigcap \overline{\cW}$, when by the definition of $\overline{W}$, the set $\tilde{A} = \{x \in \power_\alpha \alpha^{+} \mid x \cap \alpha \in A\} \in \bigcap \cW$. For every condition $p \in \mathbb{R}(\cW)$ with $\alpha\in \stem p$ there is a direct extension $q$ such that the large set associated with $\alpha$ is contained in $\tilde{A}$. In particular, $q$ forces that a tail of elements in $\overline{C} \cap \alpha$, is contained in $A$.
\end{proof}
\begin{lemma}
Every $\alpha \in \mathrm{acc}\, C_G$ is singular in $V[G]$.
\end{lemma}
\begin{proof}
First, by the factorization argument, this statement is equivalent to the statement that forcing with supercomapct Radin forcing for the coherent sequence $\cW \restriction \alpha + 1$ with top cardinal $\alpha$ and $o^\cW(\alpha) = \zeta < \alpha^{++}$ singularizes $\alpha$.

\begin{claim}
Fix $\alpha \leq \kappa$. Assume that $\zeta < \alpha^{++}$. The measures $\{ W(\alpha, \xi) \mid \xi < \zeta\}$ are discrete in the sense that there is a partition of $\power_\alpha\alpha^+$, $\langle B_i \mid i < \zeta\rangle$ such that $B_i \in W(\alpha, j)$ iff $i=j$.
\end{claim}
\begin{proof}
Let $\langle \cU_i \mid i < i_*\rangle$ be an enumeration $\langle W(\alpha, \xi) \mid \xi < \zeta\rangle$ with $i^*\leq\alpha^+$. Without loss of generality, $i^*$ is a cardinal.

For each $i < j < i_*$ let $B_{i,j} \in \cU_i \setminus \cU_j$ and $B_{j,i} = \power_\alpha\alpha^+ \setminus B_{i,j}$. Let $B_{i,i} = \power_{\alpha}\alpha^+$.

Let $B_i = \{x \in \power_\alpha\alpha^+ \mid i \in x\} \cap \triangle_{j < i_*} B_{i,j}$ where
\[\triangle_{j < i_*} B_{i,j} = \{y \in \power_\alpha\alpha^{+} \mid \forall j \in y \cap i_*,\, y \in B_{i,j}\}. \]

So, $B_i \in \cU_i$ by the normality of $\cU_i$.

Moreover, for $i < j$, $B_i \cap B_j = \varnothing$. Indeed, if $x \in B_i \cap B_j$ then $i, j \in x$ and thus $x \in B_{i,j}$ and $x \in B_{j,i}$, a contradiction.
\end{proof}
Let $h \colon i_* \to \zeta$ be the bijection used in the proof above. As the set $B_* = \bigcup_{i < \eta} B_i$ belongs to $\bigcap_{\xi < \zeta} W(\alpha,\xi)$, for any large enough $y \in C_G$, with $y \cap \kappa < \alpha$, $y \in B_*$. So, for such $y$, we can find the unique $\xi < \zeta$ such that $y \in B_i$ for $\xi = h(i)$. Without loss of generality, all the elements of the Radin club below $\alpha$ belong to $B_*$. Moreover, we may assume (by shrinking $B_i$ if necessary), that for every $x \in B_i$, $\pi_x(\bigcup_{h(\xi) < h(i)} B_\xi)$ belongs to the intersection of measures of $x \cap \kappa$, that is, $\bigcap_{\alpha<o^{\cW}(x\cap\kappa)}W(x\cap\kappa,\alpha)$.\footnote{Here we use the coherence of the sequence.}

Let us now split into cases.

\textbf{Case 0:} If $\cf \zeta < \alpha$, let $\langle \delta_i \mid i < \cf \zeta\rangle$ be a cofinal sequence at $\zeta$. Let $y_i$ be the least element in the Radin club below $\alpha$ such that $y_i\in B_{h^{-1}(\delta_i)}$ (generically, there must be such an element). If $\sup_{i < \cf \zeta} y_i\cap \kappa < \alpha$, then by the closure of $C_G$, $y_* = \bigcup y_i \in C_G$ and its intersection with $\kappa$ lies below $\alpha$. So, $y_* \in B_\rho$ for some $\rho$, but this is impossible as for all but boundedly many $i < \cf \rho$, $\delta_i \geq h(\rho)$. This is a contradiction, as by genericity, for any large enough element in $C_G$ below $y_*$  belongs to $\pi_{y_*}(\bigcup_{h(\xi) < h(\rho)} B_\xi)$ and in particular do not belong to $B_{h^{-1}(\delta_i)}$ for $\delta_i \geq h(\xi)$.

\textbf{Case 1:} if $\cf \zeta \in \{\alpha, \alpha^{+}\}$. Let $z_*$ be the element in $C_G$ with $z_* \cap \kappa = \alpha$. Let $\langle \delta_i \mid i < \cf \zeta\rangle$ be a cofinal sequence at $\zeta$. Let us shrink $B_i$ so that for all $x \in B_{h^{-1}(\delta_i)}$, $\sup (\pi_{z_*}(x) \cap \cf \zeta) > i$.

Pick $y_0 \in C_G$ arbitrary and let us recursively define $y_{n+1} \in C_G$ to be an element of $B_{h^{-1}(\delta_\xi)}$ for $\xi = \pi_{z*}(y_n \cap \cf \zeta)$.

Let us show that $\bigcup_{n<\omega} y_n = z_*$. Indeed, let $\bigcup_{n<\omega} y_n = y_*$ and let us assume that $y_* \cap \kappa < \alpha$. Then, there is $\xi$ such that $y_* \in B_\xi$.

Let $\alpha_* = y_* \cap \kappa$. Then, as before, $h(\xi) = \delta_{\alpha_*}$ (as otherwise, if $\alpha_* > h(\xi)$ there is $n < \omega$ such that $y_n \cap \kappa$ is strictly larger than $h(\xi)$ and if $\alpha_* < h(\xi)$, than for all large $n$, the $\xi'$ such that $y_n \in B_{h(\xi')}$ is bounded). But this is impossible as for all $x \in B_{\xi}$, $\sup(\pi_{z*}(x) \cap \cf \zeta) > \delta_{\alpha_*}$.
\end{proof}
The following lemma is due to Radin, \cite[Claim~8]{Radin}
\begin{lemma}
In the generic extension by $\overline{\cW}$, $\kappa$ remains measurable.
\end{lemma}

\subsection{Symmetric Model}
So, to summarize, we obtained two models: in the full Radin generic extension by $\cW$, $V[G]$, $\kappa$ is an inaccessible cardinal. Moreover, for every $\alpha < \kappa$ in the normal Radin club, $\overline{C}$, $\alpha$ is singular.

In the submodel $V[\overline{G}]$---the generic extension by the Radin club obtained from the projected measures---$\kappa$ is measurable. We would like now to consider an intermediate symmetric model, $W_1$, in which $\kappa$ remains measurable, $\overline{G}$ exists but every $\alpha$ in the Radin club is singular.

To make the definition of the automorphisms easier to understand we will adopt the conventions and represent a condition $\langle d_0,\dots, d_{n-1}, A\rangle$ as a finite sequence of pairs, $\vec{d} = \langle d'_0,\dots, d'_{n}\rangle$, whose last element is $d_n'=\langle \kappa^{+}, A\rangle$ where $A$ was previously the full measure set, and that for $d_i \in \power_\kappa\kappa^+$ will be replaced by $d_i'=\langle d_i, \varnothing\rangle$. In all other cases, $d_i = d_i'$. The stem of $p$ under this convention, therefore, is $p$ without its last coordinate.

Let us work in the symmetric model with respect to symmetries as the ones from \cite[Section~5]{CriticalCardinals}. Note, that unlike the case in \cite{CriticalCardinals}, here the filter of groups is actually $\kappa^{+}$-complete. For the completeness of this paper, let us spell out the group of automorphisms and the normal filter of groups.

\begin{definition}
  Let $g \colon \kappa^+ \to \kappa^+$ be a bijection, then $g$ lifts to $g_1 \colon \power_\kappa\kappa^+ \to \power_\kappa\kappa^+$ by $g_1(x) = g\image x$. Going further, we can lift $g_1$ to $g_2 \colon \power(\power_\kappa\kappa^+) \to \power(\power_\kappa\kappa^+)$ defined by $g_2(A) = g_1\image A$. We define $\sigma_g$ as the pointwise application of $g_1$ and $g_2$.  Namely, \[\sigma_g(\langle\langle a_0,B_0\rangle, \dots,\langle a_n,B_n\rangle\rangle)=\langle \langle g_1(a_0), g_2(B_0)\rangle, \dots, \langle g_1(a_n), g_2(B_n)\rangle\rangle.\]
\end{definition}
It is easy to verify that $\sigma_g$ is an automorphism of a dense subset of $\mathbb{R}(\cW)$, and so extends to an automorphism of the Boolean completion. So we can let $\cG$ be the group of all the automorphisms of the form $\sigma_g$ for some bijection $g\colon\kappa^+\to\kappa^+$.

\begin{definition}
For every $\alpha < \kappa^+$ let $H_\alpha$ be the group of automorphisms $\sigma_g$ for $g$ such that $g\restriction \alpha=\id$. Let $\cF = \langle\{H_\alpha \mid \alpha < \kappa^+\}\rangle$.
\end{definition}
\begin{proposition}
  $\cF$ is a normal filter of groups over $\cG$.
\end{proposition}
\begin{proof}
  Note that $\sigma_h\in\sigma_g H_\alpha\sigma_g^{-1}$ if and only if $h\restriction (g\image\alpha)=\id$. Since $\kappa^+$ is regular, let $\delta=\sup g\image\alpha$, then $H_\delta\subseteq\sigma_g H_\alpha\sigma_g^{-1}$.
\end{proof}
\begin{claim}
If $a\in C_G$, then $\power_{a\cap\kappa}a^{V[C_G]}\in W_1$.
\end{claim}
\begin{proof}
Let $a\in C_G$. Then, $\delta = \sup a \cap \kappa^{+}$ is bounded. Thus, $H_\delta$ must fix the canonical name for $a$ and any subset of $a$ will have a name fixed by $H_\delta$ as well.
\end{proof}
\begin{claim}
$\overline C\in W_1$.
\end{claim}
\begin{proof}
$\{\langle p,\check\alpha\rangle\mid \alpha\in x\in\stem p\}$ is a name for $\overline C$, and it is preserved by $H_\kappa$.
\end{proof}
In order to show that $\kappa$ remains measurable in the symmetric extension, we need to show that every symmetric subset of $\kappa$ is introduced by a small forcing over the model $V[\overline{C}]$.
\begin{lemma}
Let $X$ be a symmetric set of ordinals. Then, there is a forcing notion of cardinality $<\kappa$, $\mathbb{Q}$, such that $A$ is introduced by $\mathbb{Q}$ over $V[\overline{C}]$.
\end{lemma}
\begin{proof}
Let $\dot{X}\in\HS$ be a name for the set $A$. So, there is $\zeta < \kappa^{+}$ such that $H_\zeta$ witnesses that $\dot X\in\HS$. Without loss of generality, $\zeta \geq \kappa$.

Fix a well order of $H(\kappa^{++})$ and let $D$ be the club of all $x \in \power_{\kappa}\kappa^{+}$ which satisfy $x = \Hull^{H(\kappa^{++})}(x) \cap \kappa^+$.\footnote{For a structure $M$ with a fixed well order, we denote the Skolem hull of a set $x \subseteq M$ in $M$ by $\Hull^{M}(x)$.} Shrinking $D$ to a measure one set, we may assume that $\otp x = (x\cap \kappa)^+$ for all $x \in D$.

We claim that for every $z \in D$, such that $\zeta \in z$, $z \cap \zeta$ is fully determined by $z \cap \kappa$. Indeed, in $\Hull^{H(\kappa^{++})}(\{\zeta\})$ there is a bijection $h$ between $\kappa$ and $\zeta$, and thus as $z = \Hull^{H(\kappa^{++})}(z) \cap \kappa^{+}$, $h\image (z\cap \kappa)  = z \cap \zeta$.  


Let us isolate a forcing notion of cardinality ${<}\kappa$ that introduces $\dot{X}$ over $V[\overline{C}]$.

Fix a condition $p =\langle \vec d, \langle \kappa^{+}, A\rangle\rangle$. By shrinking the large set $A$, we may assume that $A \subseteq D \cap \{x \in \power_\kappa\kappa^{+} \mid \zeta \in x\}$.

Without loss of generality we may assume that $\len p > 0$, and let $y = a^p_{\len p - 1}$. Below $p$, the forcing is isomorphic with the product $\mathbb{R}_0 \times \mathbb{R}_1$, where $\mathbb{R}_0$ is the forcing $\mathbb{R}(\mathcal{W}\restriction \rho + 1)$ below some condition, where $\rho = y \cap \kappa$, and $\mathbb{R}_1$ is $\mathbb{R}(\mathcal{W})$ below the condition $\langle\langle \kappa^{+}, A\rangle\rangle$.

The forcing $\mathbb{R}_0$ introduces the initial segment of $\overline{C}$ up to $\rho$, so the forcing $\mathbb{Q} = \mathbb{R}_0 /(\overline{C}\restriction \rho)$ is a well defined forcing notion in $V[\overline{C}]$, introducing $\pi_y(C_G \cap y)$. Note that $|\mathbb{Q}|\leq 2^{|y|} < \kappa$. Let us show that $p$ forces that $\dot{X}$ is introduced by $\mathbb{Q}$ over $V[\overline C]$.

Let $q \leq p$ such that $q\forces\check\beta \in \dot{X}$. We claim that whenever $q' \leq p$ is a condition of the same length as $q$ and $a_i^{q'} = a_i^q$ for $a_i^q, a_i^{q'} \subseteq y$ and for every $i$, $a_i^q \cap \kappa = a_i^{q'}\cap \kappa$, it is impossible that $q' \forces\check\beta\notin \dot{X}$. This would be enough, as we can define a $\mathbb{Q}$-name $\dot{Y}$ in $V[\overline{C}]$ by \[s \forces_{\mathbb{Q}} \check \beta \in \dot{Y} \iff \exists r \in \mathbb{R}_1/(\overline{C}\restriction [\rho,\kappa)),\, (s, r) \leq p, \, (s,r)\forces_{\mathbb{R}(\mathcal{W})} \check\beta\in\dot{X}.\]
(Note that we identify $\mathbb{R}_0\times\mathbb{R}_1$ with $\mathbb{R}(\cW)$ below $p$.)

Indeed, by extending $q$ and $q'$ if needed, we may assume that \[q = \langle d_0,\dots, d_{n}\rangle, q' = \langle d'_0,\dots, d'_{n}\rangle\] with $d_i = \langle y_i, B_i\rangle$, $d_i' = \langle y_i', B_i'\rangle$. For some $k<n$, $y_k=y$, and so for all $j<k$, $y_j=y_j'$ and for every $j\geq k$, $y_j\cap\kappa=y_j'\cap\kappa$. Given such $j\geq k$, we have that $\zeta\in y_j\cap y_j'$, and so $\zeta\cap y_j=\zeta\cap y_j'$.

We can now define an automorphism $\sigma_f$ sending the stem of $q$ to the stem of $q'$, similarly to the automorphism defined in \cite{CriticalCardinals}. By induction on $j \geq k$, let us define a bijection $f_j \colon y_j \to y_j'$ such that $f_k = id_{y_k}$ and $f_j\restriction \zeta$ is the identity. Let as assume that $f_{j-1}$ is defined and let us define $f_j$. As $\otp y_j=\otp y_j' = (y_j\cap \kappa)^+$, $|y_j \setminus y_{j-1} \setminus \zeta| = |y'_j \setminus y'_{j-1} \setminus \zeta|$. Indeed, the cardinality of $y_j\cap \zeta$ is $y_j\cap \kappa$ and the cardinality of $y_{j-1}, y'_{j-1}$ is strictly smaller. So, we can find $f_j$ extending $f_{j-1}$, as needed. Finally, let $f = f_n$ and by the construction $\sigma_f\in H_\zeta$ and $\sigma_f(q)$ is compatible with $q'$. Therefore they must agree on the truth value of $\check\beta\in\dot X$, as both names are preserved by $\sigma_f$.\footnote{Here the main advantage of using coherent sequence instead of measure sequence manifests itself. In order to make the stems compatible, it is make the sets of $\power_\kappa\kappa^{+}$ on the stems identical, and we do not need to worry about the choice of the measure sequences.}
\end{proof}
As every set of ordinals is added by a small forcing, any normal measure in $V[\overline C]$ on $\kappa$ extends to a measure in $W_1$, as shown by Jech in \cite{Jech}.

\begin{lemma}
  $W_1$ is $\kappa$-closed in $V[C_G]$.
\end{lemma}
\begin{proof}
Suppose that for some $\delta<\kappa$ and $f\in V[C_G]$, $f\colon\delta\to W_1$. We will show that $f\in W_1$ as well. Fix $p\in G$. In order to construct a symmetric name for $\dot f$, we need to show that the sequence of \emph{names} $\dot f(\alpha)$ is equivalent to a symmetric name.

\begin{claim}
There is an ordinal $\xi < \kappa$ and a function $F$ in $V$, with domain $\xi$ and range contained in the symmetric names, such that in $V[C_G]$ there is a set $c \subseteq \xi$ satisfying $(F\image c)^G = \dot{f}^G$.
\end{claim}
\begin{proof}
Using the chain condition, for every $\alpha < \delta$, we can find a collection $S_\alpha$ of $\leq \kappa^{+}$ many pairs of the form $\langle \alpha,\dot{x}\rangle$ where $\dot{x}$ is a symmetric name and $p$ forces that $\langle \alpha, \dot{f}^G(\alpha)\rangle$ equals to an evaluation of a member of $S_\alpha$.

Let $F_0$ be a function with domain $\kappa^{+}$ covering $\bigcup_{\alpha < \delta} S_\alpha$. In $V[C_G]$ we can find a set $a_0$ of size $<\kappa$ such that $F_0\image a_0$ already contains for each $\alpha < \delta$ a name which is going to be evaluated as $f(\alpha)$. As $\cf^{V[C_G]} \kappa^{+} = \kappa$, $a_0$ is bounded, by some ordinal $\zeta_0$. Extending $p$ if necessary, we may decide the value of $\zeta_0$. Let $H\colon \kappa \to \zeta_0$ a surjection and let $F_1 = F_0 \circ H$, $c = H^{-1}(a_0)$. Applying the same argument for $\kappa$, which remains regular in $V[C_G]$, $\sup c < \kappa$ and by extending $p$ again, we may decide it to be some $\xi$. Finally, let $F = F_1 \restriction \xi$.
\end{proof}

Fix a name $\dot{c}$ for the set obtained in the claim.

Using \cite[Lemma~3.1]{Krueger}, if $\xi < a\cap\kappa$ for every $\langle a,A\rangle\in\stem p$, we can find a descending sequence of direct extensions $p_\alpha$ for $\alpha\leq\xi$. Thus, we can decide the value of the set $\dot{c}$ from the claim using a sequence of $\xi$ direct extensions. As this can be done densely below $p$, it means that $f\in W_1$ as wanted.

If, however, it is not the case that $\xi<a\cap\kappa$ for all $a$ in the stem of $p$, then by applying Lemma~4.1 and Lemma~5.8 from \cite{Krueger} we can split $p$ into two parts $p^{\leq m}$ and $p^{>m}$, factorize $\mathbb{R}(\cW)\restriction p=\mathbb R(\cW)\restriction p^{\leq m}\times\mathbb R(\cW)\restriction p^{>m}$ and make the above argument in $\mathbb R(\cW)\restriction p^{>m}$. So, the value of $\dot{c}^G$ can be computed using the generic for the lower part, $\mathbb R(\cW)\restriction p^{\leq m}$, $G_{low}$. As $G_{low}$ is symmetric, we conclude that there is a symmetric name equivalent to $\dot{c}$ and thus also the set of names $F \image \dot{c}$ is symmetric.

Now, in $V$, for each $\alpha < \xi$, let $s(\alpha)=\beta$ be the least ordinal such that $H_\beta \subseteq \sym(F(\alpha))$, and let $\zeta < \kappa^{+}$ be large enough so that $s\image \xi \subseteq \zeta$. So, every automorphism from $H_\zeta$ that fixed the generic for $\mathbb R(\cW)\restriction p^{\leq m}$ fixes every member of $F\image \dot{c}$.

Thus, we obtain that $W_1$ is $\kappa$-closed in $V[C_G]$, as wanted.
\end{proof}
This lemma provides us with two important corollaries.
\begin{corollary}
  $W_1\models\DC_{<\kappa}$.
\end{corollary}
\begin{proof}
By \cite[Lemma~3.2]{Karagila:DC}, since $V[C_G]$ is a model of $\ZFC$ and $W_1$ is $\kappa$-closed, $W_1\models\DC_{<\kappa}$.
\end{proof}
\begin{corollary}
  $V_\kappa^{V[C_G]}=V_\kappa^{W_1}$. In particular, $V_\kappa^{W_1}\models\ZFC$.\qed
\end{corollary}
\section{The least inaccessible cardinal}
Let us strengthen \cref{thm:non-mahlo} by collapsing cardinals below $\kappa$ to make it the least inaccessible.
\begin{theorem}\label{thm:first-inaccessible}
There is a symmetric extension of $W_1$ where $\kappa$ remains measurable and is the least inaccessible cardinal. In particular, if $\GCH$ holds and $\kappa$ is a $\kappa^{++}$-supercompact, then there is a symmetric extension in which $\kappa$ is a measurable cardinal which is the least inaccessible cardinal.
\end{theorem}
\begin{proof}
  Recall that every successor point in $\overline C\subseteq\kappa$, the Radin club, is regular and every limit point is singular in $W_1$. Since $V[\overline C]\subseteq W_1$, we can use that to define the symmetric extension. Let $C=\overline C\cup\{\omega\}$ be enumerated as $\{\rho_\alpha\mid\alpha<\kappa\}$, and define $\mathbb{P}$ to be the Easton-support product $\prod_{\alpha<\kappa}\Col(\rho_\alpha^+,{<}\rho_{\alpha+1})$. Note that $\mathbb{P}\subseteq V_\kappa^{W_1}=V_\kappa^{V[C_G]}$, so all of its initial segments are well-orderable, and behave as expected. We will also write $\mathbb P_{\leq\delta}$ ($\mathbb P_{<\delta}$) and $\mathbb P^{>\delta}$ ($\mathbb P^{\geq\delta}$) to indicate the factorization of $\mathbb P$ into the initial segment of the product up to $\delta$ and its remainder.

  Let us claim first that in the generic extension the cardinals below $\kappa$ are exactly the cardinals $\{\rho_{\alpha} \mid \alpha < \kappa\} \cup \{(\rho_{\alpha}^{+})^{W_1} \mid \alpha < \kappa\}$.
  Clearly, every other cardinal is collapsed, so we need to show that those cardinals are not collapsed. Let $\mu = \rho_\alpha^{+}$. If $\mu$ is collapsed, then its cofinality must be $\leq \rho_\alpha$. If $\rho_\alpha$ is regular, then $\alpha$ is non-limit and a successor ordinal. So the forcing $\mathbb{P}_{<\alpha}$ is $\rho_\alpha$-c.c. If $\alpha$ is a limit ordinal, then $\rho_\alpha$ is singular in $W_1$ and thus the new cofinality of $\mu$ must be strictly below $\rho_\alpha$. Fix $\beta < \alpha$ successor ordinal such that $\cf \mu \geq \rho_{\beta}^+$. Now, the forcing $\mathbb P^{\geq \beta}$ is $\rho_{\beta}^{+}$-closed and thus cannot change $\mu$-s cofinalty to be less than $\rho_\beta^{+}$. The forcing $\mathbb P_{\leq\beta}$ $\rho_{\beta}$-c.c.\ and thus cannot change the cofinality of $\mu$ over the generic extension by $\mathbb P^{\geq \beta}$, as it is $\geq \rho_{\beta}^{+}$ as this intermediate model.

  Our group, $\cG$, is the Easton-support product of $\operatorname{Aut}(\Col(\rho_\alpha^+,{<}\rho_{\alpha+1}))$, acting pointwise on $\mathbb P$. The filter is generated by $\fix(\alpha)=\{\pi\in\cG\mid\pi\restriction\mathbb P_{<\alpha}=\id\}$ for $\alpha<\kappa$. Let $W_2$ be the symmetric extension of $W_1$, given by some generic filter.

  It is a standard argument that $\mathbb P$ is homogeneous, that every proper initial segment of the generic is hereditarily symmetric, and that every set of ordinals in the symmetric extension is added by a proper initial segment of the generic. In particular, Jech's theorem applies and $\kappa$ remains measurable in $W_2$.

  Let $\alpha<\kappa$ be an uncountable regular cardinal in $W_2$. As the regular cardinals in $W_2$ below $\kappa$ are the same as the regular cardinals in the generic extension of $W_1$, no uncountable regular cardinal below $\kappa$ is a limit cardinal, so $\kappa$ is the least inaccessible cardinal.
\end{proof}

\begin{theorem}
$W_2\models\DC_{<\kappa}$.
\end{theorem}
\begin{proof}
  Working in $W_1$, given any successor ordinal $\alpha<\kappa$ let $\delta=\rho_\alpha$. Decompose $\mathbb P$ into $\mathbb{P}_{<\alpha}\times\mathbb{P}^{\geq\alpha}$. Then $W_1\models|\mathbb{P}_{<\alpha}|=\delta$ and $\mathbb P^{\geq\alpha}$ is $\delta^+$-closed. Moreover, we can naturally decompose the symmetric system itself into a product of symmetric systems given by these two component. Since $\mathbb P_{<\alpha}$ is fixed pointwise, the restriction of the generic filter to that part belongs to the symmetric extension $W_2$. Therefore $W_2$ is the \emph{generic} extension of $W_{2,\alpha}$, the symmetric extension given by $\mathbb P^{\geq\alpha}$ component.

  Since $W_1\models\DC_\delta$ and $\mathbb P^{\geq\alpha}$ is a $\delta^+$-closed with the filter being $\delta^+$-complete, we get by \cite[Lemma~3.1]{Karagila:DC} that $W_{2,\alpha}\models\DC_\delta$. Finally, by \cite[Theorem~2.1]{GitmanJohnstone}, we get that $W_2\models\DC_\delta$ as well. As this holds for unboundedly many $\delta<\kappa$, $W_2\models\DC_{<\kappa}$.
\end{proof}
\section{Lower bounds on the consistency strength}
As with many similar results, one is left to wonder if the use of supercompactness is truly necessary, at least in terms of consistency strength. The trivial lower bound of a single measurable cardinal was improved by the second and third authors in \cite[Theorem~3.6]{SmallMeasurableCardinals} to show that in the core model, a cardinal with Mitchell order $2$ must exist. In this section we improve this result.

Throughout this section, $o(\alpha)$ denotes the Mitchell order of $\alpha$ and $K$ is Mitchell's core model for the anti-large cardinal hypothesis ``there is no $\alpha$ such that $o(\alpha) = \alpha^{++}$''. \footnote{Slightly modified argument can be phrased under the more permissive hypothesis of ``there is no inner model with a strong cardinal'', but as our current result is much weaker than that, there is no need to weaken the hypothesis in this direction.} The proof of the theorem relies heavily on Mitchell's covering lemma \cite[Theorem~4.19]{MitchellHandbook}. We assume that the reader is familiar with the basic definitions and theorems of \cite{MitchellHandbook}.

\begin{theorem}[$\ZF$]\label{thm:lower-bound}
  Let $\kappa$ be a strongly inaccessible non-Mahlo measurable cardinal, then $K^{\HOD}\models o(\kappa)\geq\kappa+1$.
\end{theorem}
\begin{proof}
Let $\kappa$ be a measurable cardinal, and let $U$ be a $\kappa$-complete ultrafilter on $\kappa$.\footnote{By \cite{BilinskyGitik}, it might be that $\kappa$ does not carry any normal measures, which is why we cannot assume $U$ is normal.} Let $C \subseteq \kappa$ be a club of singular cardinals.

\begin{lemma}
$K^{\HOD} = K^{\HOD[C]} = K^{\HOD[C][U \cap \HOD[C]]}$.
\end{lemma}
\begin{proof}
By Vop\v{e}nka's theorem \cite{Vopenka} (see also \cite[Theorem~15.46]{JechSetTheory}), every set or ordinals is generic over $\HOD$. Since $K$ is generically absolute, $K^{\HOD[C]} = K^{\HOD}$. Applying this to the set $U \cap \HOD[C]$ (or rather to a set of ordinals encoding it, which exists as $(\power(\kappa))^{\HOD[C]}$ is well orderable) we obtain the second equality.
\end{proof}

Let $K = K^{\HOD}$. Let $M=\HOD[C]$. The ultrafilter $U$ measures every set in $M$ and thus we can define in $V$ and elementary embedding $j \colon M \to N$, with critical point $\kappa$ (using the fact that $M$ is a model of $\ZFC$). Using this $j$ we can derive an $M$-normal measure on $\kappa$, $D = \{A \subseteq \kappa \mid A \in M,\, \kappa \in j(A)\}$, containing every club from $M$.

Next, since $\kappa$ is regular, $N \models``\kappa$ is regular''. Therefore, the set of all $M$-regular cardinals below $\kappa$ must be in $D$ and in particular, $C$ must contain cardinals which are regular in $M$ and thus in $K$, but singular in $V$.

Finally, since $D\in M[U\cap M]$, by the maximality of $K$, the $K$-normal measure $D \cap K$ belongs to $K$.

Let us denote by $A = \Reg^M\cap \kappa = \{\zeta < \kappa \mid M\models \zeta\text{ is regular}\}$. In \cite{SmallMeasurableCardinals} the argument was that if $\zeta\in A\cap C$, it must be measurable in $K$. The reason is that $\zeta$ is singular in $V$, so we can find some $t\subseteq\zeta$ witnessing that and add it generically to $M$. Since, as in the lemma above, $K^{M[t]}=K$, by the covering lemma we get that $\zeta$ must be measurable in $K$. So, the set of all $K$-measurable cardinals belongs to $D$ and thus $K\models o(\kappa) \geq 2$. By conducting a much more careful analysis of the covering models of $K$ we will see that $o(\kappa)$ must be much higher.

We define a sequence of clubs:
\begin{enumerate}
\item $C_0 = C$,
\item $C_{\alpha + 1} = \acc C_\alpha$,
\item for limit $\alpha$, $C_\alpha = \bigcap_{\zeta < \alpha} C_{\zeta}$.
\end{enumerate}
So $\langle C_\alpha \mid \alpha < \kappa\rangle$ is the sequence of derivatives of $C$ up to $\kappa$.
\begin{lemma}
For every $\zeta \in C_\alpha \cap A$, $K\models o(\zeta) \geq \alpha$.
\end{lemma}
\begin{proof}
We prove the lemma by induction on $\alpha$. For $\alpha = 0$ the claim is trivial, and for limit $\alpha$ it follows easily from the definition of $C_\alpha$.

Let us argue for successor ordinal. Let $\zeta \in C_{\alpha+1} \cap A$ (there is such $\zeta$, since as a club $C_{\alpha+1}\in D$). In particular, $\zeta \in C$ and thus it is singular in $V$. Let $t \subseteq \zeta$ be a cofinal sequence witnessing it with $\otp t = \cf^V \zeta$, as before $K^{M[t]} = K$.

In $M[t]$, let $W \prec H(\lambda)$, for some $\lambda > \kappa$, with $|W|<\zeta$ containing $\alpha \cup \{t, C\}$ (in particular, $C_\alpha \in W$). By Mitchell's covering lemma \cite[Theorem~1.8]{MitchellHandbook} there is a weak Prikry-Magidor sequence $I \subseteq \zeta$ and $g \colon \zeta \to \zeta$ in $K$ such that $W \cap \zeta = \bigcup_{\xi \in I} (g(\xi) \setminus \xi)\cup \rho$ for some $\rho < \zeta$. By replacing $g$ with $\alpha\mapsto \sup \{g(\beta) \mid \beta \leq \alpha\} + \alpha$, if necessary, we may assume that $g$ is strictly increasing. By the definition of a weak Prikry-Magidor sequence and by shrinking $I$, if necessary, we may also assume that $\forall\xi< \zeta, g(\xi) < \min (I\setminus (\xi + 1))$.

As the function $g$ belongs to $K$, the club, $D_g = \{\eta < \zeta \mid g\image \eta \subseteq \eta\}$ is in $K$ and in particular, in $M$. Since $\zeta$ is regular in $M$, $C_\alpha \cap D_g$ is still a club of order type $\zeta$.

If $o(\zeta) \leq \alpha < \zeta$, then for almost every $c\in I$, $o(c) < \alpha$, by \cite[Theorem~1.5]{MitchellHandbook}. We will show that this is not the case.
\begin{claim}\label{claim:successors-points-of-I}
If $c$ is a successor in $I$ and a singular ordinal in $M$, then $o(c) \geq \alpha$.
\end{claim}
\begin{proof}
As $I$ is a weak Prikry-Magidor sequence, almost every element of $I$ is regular in $K$ using \cite[Theorem~1.7]{MitchellHandbook}, taking the large set of $K$-regular cardinals. If $c$ is singular in $M$, then $W$ must contain a sequence $s$ witnessing that and by its regularity in $K$, we may assume that $s$ is a weak Prikry-Magidor sequence as well. If $o(c) \leq \alpha$, then $\otp s \leq \omega^\alpha$, using Theorem~1.7 of \cite{MitchellHandbook} and induction on the Mitchell order. Thus it is fully contained in the model $W$. As the weak Prikry-Magidor sequence itself is $g \restriction c$-indiscernible (as it eventually enters every $K$-club and $g \restriction c$ is a function from $c$ to itself), we conclude that $I \cap c$ is cofinal in $c$.
\end{proof}

So, without loss of generality, any $c \in I\setminus\acc I$ is regular in $M$. We will show that cofinally many successor point $c\in I$ are in $C_\alpha$, and thus $o(c) \geq \alpha$.

Let $\gamma < \zeta$ be an arbitrary ordinal. Since $|C_\alpha \cap D_g|^M=\zeta > |I|^M$, we may pick $\delta\in (C_\alpha \cap D_g \setminus \gamma) \setminus I$ and $c = \min (I \setminus \delta)$. So, $c$ must be a successor element of $I$ and thus (by our assumption) regular in $M$. Let us show that $c$ must be in $C_\alpha$.

Otherwise, let $\eta = \max(C_\alpha \cap c) \in W$. Clearly, $\eta \geq \delta$, as $\delta \in C_\alpha$. By the defining property of $g$, if $\gamma$ is sufficiently large, $\eta \in \bigcup_{u\in I} g(u)\setminus d$. So, there must be $u < c$ in $I$ such that $\eta < g(u)$. Since $\delta \notin I$, $u < \delta$. But $\delta$ belongs to $D_g$, and so $g(u) < \delta \leq \eta$, a contradiction.

So, $c \in C_\alpha \cap A$ and by the inductive hypothesis, $K\models o(c) \geq \alpha$.
\end{proof}
This implies that $K\models o(\kappa)\geq\kappa$. In order to show that $K\models o(\kappa) \geq \kappa+1$, it is enough to show that $\Ult(M,U)\models o^K(\kappa) \geq \kappa$. Indeed, in this model $\kappa$ is regular and belongs to $j(C_\alpha)$ for every $\alpha < \kappa$, so in $\Ult(M,U)$, $o^K(\kappa)\geq\kappa$.
\end{proof}

\begin{remark}
Why are we not continuing the proof to obtain $o(\kappa)\geq\kappa+2$, or even higher? Unfortunately, the above proof will fail. The proof of \cref{claim:successors-points-of-I} relies on the fact that $\alpha$ is covered by the small model, $W$. Once $o(\kappa)\geq\kappa$ is reached, the claim can no longer work, since the covering model is not small. We are then allowed one more step, to obtain $o(\kappa)\geq\kappa+1$ by using the ultrapower by $U$.
\end{remark}

\begin{question}
\begin{enumerate}
\item Can the lower bound be improved?
\item What is the exact consistency strength of a strongly inaccessible non-Mahlo measurable cardinal?
\end{enumerate}
\end{question}

The proof, as written here, depends on the fact that $\kappa$ is Mahlo in $M$ and there is a club $C \in M$ such that every element of $C$ can be singularized in a generic extension of $M$. By a slight modification of the argument, one can provide the following better formulation.
\begin{theorem}
Assume that there is no inner model of $\exists \alpha,\, o(\alpha) = \alpha^{++}$. Let $\kappa$ be an inaccessible cardinal such that there is a club $C\subseteq \kappa$ through the singular cardinals below $\kappa$ and there is a forcing notion singularizing $\kappa$ while preserving strong limitness. Then $K\models o(\kappa) \geq \kappa$.
\end{theorem}
\begin{proof}
Let $t$ be a short cofinal sequence at $\kappa$ and let $W \prec H(\lambda)$ contain $t, C$ and $\alpha$ for some $\alpha < \kappa$. Let us assume, towards a contradiction, that $o(\kappa) < \alpha$.

Let $I$ be a weak Prikry-Magidor sequence for $W \cap \kappa$ as before, so \[W \cap \kappa \subseteq \rho \cup \bigcup_{\zeta \in I} g(\zeta) \setminus \zeta,\, \rho < \kappa,\, g\in K.\] By repeating \cref{claim:successors-points-of-I}, a successor element of $I$ must either be of Mitchell order above $\alpha$, or regular in $M$.

Let us argue that there are unboundedly many elements of $C$ are successor elements of $I$, so they are singular in $M$. Indeed, exactly as before, for arbitrary $\gamma$ we pick $\delta \in (C \cap D_g \setminus \gamma) \setminus I$ and set $c = \min (I \setminus \delta)$. We claim that $c\in C$.

Otherwise, let $\eta = \max (C \cap c)$ and even though $\eta \in W$, $\eta$ is not covered by $g$ using any smaller element of $I \cap c$ and $\rho$, as $\sup (I \cap c) \leq\delta \leq \eta$, and the supremum is not obtained.
\end{proof}

In \cite{BGH}, the problem of the consistency strength of embedding the forcing for shooting a club through the singular cardinals into a tree Prikry forcing was studied. The consistency strength of this situation was bounded from below by $o(\kappa) \geq \kappa^{+}+1$ and from above by a slight strengthening of a superstrong cardinal. Even though it is unclear whether the construction of \cite{BGH} can be used in our situation, it seems reasonable that a similar method might be used in order to obtain a model with a non-Mahlo measurable cardinal. Thus we conjecture that the consistency strength of \cref{thm:non-mahlo} can be further reduced and that the lower bounds for \cref{thm:lower-bound} are not optimal either.
\section*{Acknowledgements}
The authors would like the thank the anonymous referee for their careful reading of this paper and their helpful remarks. The first author was partially supported by ISF grant No 882/22. The second author was partially supported by the Israel Science Foundation grant 1967/21. The third author was supported by UKRI Future Leaders Fellowship [MR/T021705/2]. No data are associated with this article.
\bibliographystyle{amsplain}
\providecommand{\bysame}{\leavevmode\hbox to3em{\hrulefill}\thinspace}
\providecommand{\MR}{\relax\ifhmode\unskip\space\fi MR }
\providecommand{\MRhref}[2]{%
  \href{http://www.ams.org/mathscinet-getitem?mr=#1}{#2}
}
\providecommand{\href}[2]{#2}

\end{document}